\documentclass[11pt]{article}

\usepackage{amssymb}
\usepackage{amsmath}
\usepackage{mathrsfs}
\usepackage{amsfonts}
\usepackage{vmargin}
\usepackage{amsthm}
\usepackage{graphicx}
\usepackage{authblk}
\usepackage{xfrac}
\usepackage[T1]{fontenc}

\long\def\symbolfootnote[#1]#2{\begingroup%
\def\thefootnote{\fnsymbol{footnote}}\footnote[#1]{#2}\endgroup}

\setmarginsrb{20mm}{20mm}{20mm}{20mm}{10mm}{10mm}{10mm}{10mm}

\newtheoremstyle{remark}
  {}{}{}{}{\bfseries}{.}{.5em}{{\thmname{#1 }}{\thmnumber{#2}}{\thmnote{ (#3)}}}

\RequirePackage{amsthm}

\theoremstyle{definition}
\newtheorem{defi}{Definition}[section]
\newtheorem{rem}[defi]{Remark}
\theoremstyle{plain}

\newtheorem{tw}[defi]{Theorem}

\newtheorem{lem}[defi]{Lemma}

\newtheorem{prop}[defi]{Proposition}

\def\vint{\mathop{\mathchoice%
          {\setbox0\hbox{$\displaystyle\intop$}\kern 0.22\wd0%
           \vcenter{\hrule width 0.6\wd0}\kern -0.82\wd0}%
          {\setbox0\hbox{$\textstyle\intop$}\kern 0.2\wd0%
           \vcenter{\hrule width 0.6\wd0}\kern -0.8\wd0}%
          {\setbox0\hbox{$\scriptstyle\intop$}\kern 0.2\wd0%
           \vcenter{\hrule width 0.6\wd0}\kern -0.8\wd0}%
          {\setbox0\hbox{$\scriptscriptstyle\intop$}\kern 0.2\wd0%
           \vcenter{\hrule width 0.6\wd0}\kern -0.8\wd0}}%
          \mathopen{}\int}
\def\={\hspace{-3mm}&=&\hspace{-3mm}}

\newcommand{\ddt}{\frac{{d}}{dt}}
\newcommand{\iab}{\int_a^b}
\let \epsilon \varepsilon
\let \phi \varphi
\newcommand{\lab}[1]{L^{#1}(a,b)}
\newcommand{\hab}{H^1(a,b)}
\newcommand{\habo}{H^1_0(a,b)}
\newcommand{\iot}{\int_0^T}

\begin{document}
\date{}
\title{\bf Global-in-time regular unique solutions with positive temperature to the 1d thermoelasticity}

\author[1]{Piotr Micha{\l} Bies\footnote{piotr.bies@pw.edu.pl}} \author[2]{Tomasz Cie\'slak\footnote{cieslak@impan.pl}}
\affil[1]{Faculty of Mathematics and Information Sciences, Warsaw University of Technology, Ul. Koszykowa 75, 00-662 Warsaw, Poland}
\affil[2]{Institute of Mathematics, Polish Academy of Sciences, Ul. \'Sniadeckich 8, 00-656 Warsaw, Poland}
\maketitle

\begin{abstract}
We construct a unique regular solution to the minimal nonlinear system of the 1d thermoelasticity. The obtained solution
has a positive temperature. Our approach is based on an estimate, using the Fisher information, which
seems completely new in this context. It is combined with a recent inequality in \cite{CFHS} and embeddings,
which allows us to obtain a new energy estimate. The latter is used in a half-Galerkin procedure yielding
global solutions. The uniqueness and further regularity of such solutions are obtained.
\end{abstract}

\bigskip

\medskip


\medskip
\section{Introduction}
The present paper is devoted to the study of the one-dimensional minimal hyperbolic-parabolic system of thermoelasticity.
In the next section, we present the physical derivation of the system and discuss assumptions leading to the following system of equations. Our main result is a global-in-time existence of a unique regular solution to the obtained system.

Let us take $a,b\in\mathbb{R}$ such that $a<b$ and $T>0$. We are interested in real-valued functions of time and space $\theta$ and $u$ satisfying
\begin{equation}\label{eq}
\begin{cases}
u_{tt}-u_{xx}=\mu\theta_x,&\textrm{in }(0,T)\times (a,b),\\
\theta_t-\theta_{xx}=\mu\theta u_{tx},&\textrm{in }(0,T)\times (a,b),\\
u(\cdot,a)=u(\cdot,b)=\theta_x(\cdot,a)=\theta_x(\cdot,b)=0,&\textrm{on } (0,T),\\
u(0,\cdot)=u_0,\ u_t(0,\cdot)=v_0,\ \theta(0,\cdot)=\theta_0>0,
\end{cases}
\end{equation}
where $\mu$ is a constant, initial data $\theta_0, u_0, v_0$ are given. The regularity of initial data is discussed later.
Functions $\theta$ and $u$ are interpreted as temperature and displacement of the 1D material (a rod for instance) respectively.
We consider the evolution of a rod undergoing heat displacement.

Notice that for the regular functions $\theta$ and $u$ one has the conservation of energy. Indeed, the following simple proposition holds.
\begin{prop}\label{energy_cons}
Regular solutions $\theta, u$ of \eqref{eq} satisfy
\begin{align}\label{balance}
\frac12\iab u_t^2dx+\frac12\iab u_x^2dx+\iab\theta dx=\frac12\iab v_0^2dx+\frac12\iab u_{0,x}^2dx+\iab\theta_0 dx.
\end{align}
\end{prop}
\begin{proof}
We multiply the first equation of \eqref{eq} by $u_t$ and integrate over $(a,b)$. It yields
\begin{align}\label{bal1}
\frac12\ddt\left(\iab u_t^2dx+\iab u_x^2dx\right)=\mu\iab\theta_x u_tdx.
\end{align}
Then, we integrate the second equation. We get
\begin{align}\label{bal2}
\ddt\iab\theta dx=-\mu\iab\theta_x u_tdx.
\end{align}
Now, we add (\ref{bal1}) and (\ref{bal2}) to obtain
\begin{align*}
\ddt\left(\frac12\iab u_t^2dx+\frac12\iab u_x^2dx+\iab\theta dx\right)=0.
\end{align*}
And so the total energy balance \eqref{balance} is obtained.
\end{proof}

In the present paper, we shall obtain the existence of regular global-in-time a unique solution to \eqref{eq} such that the temperature $\theta$ is positive. Below we formulate the definition of the solution.
\begin{defi}\label{weakdef}
We say that $(u,\theta)$ is a solution to (\ref{eq}) if:
\begin{itemize}
\item The initial data is of regularity
\begin{align*}
u_0\in H^2(a,b)\cap\habo,\quad v_0\in\habo,\quad \theta_0\in\hab.
\end{align*}
We also require that there exists $\tilde{\theta}>0$ such that $\theta_0(x)\geq\tilde{\theta}$ for all $x\in[a,b]$.
\item Solutions $\theta$ and $u$ satisfy
\begin{align}
u&\in L^{\infty}(0,T;H^2(a,b))\cap W^{1,\infty}(0,T,\habo)\cap W^{2,\infty}(0,T;\lab{2}),\nonumber\\
\theta&\in L^{\infty}(0,T;\hab)\cap H^1(0,T; \lab{2}).\label{uniq}
\end{align}
\item The momentum equation
\begin{align*}
 u_{tt}- u_{xx}=\mu\theta_x
\end{align*}
is satisfied for almost all $(t,x)\in (0,T)\times (a,b)$.
\item The entropy equation
\begin{align*}
\iab\theta_t\psi dx+\iab\theta_x\psi_x dx =\mu\iab\theta u_{tx}\psi dx
\end{align*}
holds for all $\psi\in C^{\infty}[a,b]$ and for almost all $t\in (0,T)$.
\item Initial conditions are attained in the following sense:
\[
\theta \in C([0,T];L^2(a,b)),
\]
\[
u\in C([0,T];\habo),\ u_t \in C([0,T];L^2(a,b)).
\]
\end{itemize}
\end{defi}
The following result is the main achievement of the paper.
\begin{tw}\label{main_theorem}
For any $T>0$ and any constant $\mu$ there exists a unique solution of \eqref{eq} in the sense of Definition \ref{weakdef} with a positive temperature.
\end{tw}
As far as we know this is the first global-in-time existence and uniqueness result independent of the size of initial data for the minimal
thermoelasticity model. The minimal model is introduced as a particular case of the general 1D thermoelasticity model considered for instance in \cite{slemrod}:
\begin{equation}\label{general}
\begin{cases}
u_{tt}-\alpha(\theta, u_x)u_{xx}=\beta(\theta, u_x)\theta_x,&\textrm{in }(0,T)\times (a,b),\\
\gamma(\theta, u_x)\theta_t-\delta(\theta_x)\theta_{xx}=\beta(\theta, u_x)\theta u_{tx},&\textrm{in }(0,T)\times (a,b),\\
u(\cdot,a)=u(\cdot,b)=\theta_x(\cdot,a)=\theta_x(\cdot,b)=0,&\textrm{on } (0,T),\\
u(0,\cdot)=u_0,\ u_t(0,\cdot)=v_0,\ \theta(0,\cdot)=\theta_0>0.
\end{cases}
\end{equation}
The physical meaning of functions $\alpha,\beta,\gamma$ and $\delta$ is discussed in Section \ref{sec2}. Here we only notice that our minimal model appears
as a particular case of the general model when all the thermodynamical parameter functions $\alpha,\beta,\gamma$ and $\delta$ equal one. A minimal model has been studied for example in \cite{TC1, TC2}. In \cite{TC1} the local existence of a regular unique solution to \eqref{eq} with positive temperature is obtained. Some very weak measure valued global solution is also constructed. In \cite{TC2} the global-in-time solution in the 3D case of \eqref{eq} is obtained, the construction applies also to the 2D case. Its temperature is a positive measure and obtained weak solution with defect measure satisfies the weak-strong uniqueness. However, the present result is the first to expose the regular global-in-time solution. Moreover, this solution is unique and the temperature involved is positive.

In the case of the studies of more general systems of the form \eqref{general} no global-in-time solutions for the arbitrary size of initial data are known, even for quite regular functions $\alpha,\beta,\gamma,\delta$. What is known, is the global regular solution for small initial data (see \cite{hrusa_tarabek, racke, racke_shibata, slemrod}). In \cite{rsj} the global solution is obtained for initial data close to the equilibrium. Also, the positivity of temperature is an issue and the constructed solutions were obtained often under the assumption $\beta>0$, which excludes the maximum principles in the heat equation and hence the positivity of temperature.

A slightly different approach was taken in \cite{daf_chen}, where the solution is obtained via the vanishing viscosity method to the 1D thermoelasticity model similar to \eqref{general}, but not fully equivalent. Moreover, for such models finite-time blowups of solutions were obtained, see \cite{daf_hsiao} or \cite{hrusa_m}. An approach via the hyperbolic laws formulation was also recently invoked in \cite{galanopoulou}.

Our result, obtained for the minimal model \eqref{eq}, makes use of the type of estimate which, according to our knowledge, is completely new in this area. We find a new energy-like functional built upon the Fisher information of temperature and corresponding higher derivatives of displacement and its velocity. The above-mentioned crucial formula is derived formally in Section \ref{sec3}. Next, embeddings and a recent functional inequality in \cite{CFHS} allow us to close a priori estimates. Those, applied to the half-Galerkin approximation, suggested in \cite{TC2}, yield the described solution (see Section \ref{sec4}). Section \ref{sec5} is devoted to the proof of the uniqueness.

In the last section, we discuss the potential applicability of our method in the general 1D model \eqref{general} as well as in the higher dimensional version of \eqref{eq}.

Last, but not least, let us state a special 1D case of the functional inequality obtained recently in \cite{CFHS}. It turns out very useful
when dealing with the Fisher information and its dissipation.
\begin{lem}\label{cfhslem}
Let us consider a positive function $\psi\in C^2([a,b])$, $\psi_x(a)=\psi_x(b)=0$, then the following inequality
$$
\iab\left[\left(\psi^{\sfrac12}\right)_{xx}\right]^2dx\leq \frac{13}{8}\iab\psi\left[\left(\log\psi\right)_{xx}\right]^2dx
$$
holds.
\end{lem}
\section{Physical derivation}\label{sec2}
In the present article, we discuss a brief derivation of the physical foundations of a system \eqref{eq} and discuss if the assumptions leading to it are physically relevant. For a more detailed explanation, we refer to \cite{TC2} or to \cite{slemrod}.

System \eqref{eq} is a one-dimensional version of the mechanical description of the balance of momentum as well as the balance of internal energy, under some assumptions on the constitutive relations. The general system reads
\begin{equation}
\begin{cases}
\rho u_{tt}=\nabla \cdot \sigma,&\textrm{in }(0,T)\times \Omega,\\
e_t+\nabla\cdot q=\sigma\cdot \operatorname{div} u_t,&\textrm{in }(0,T)\times \Omega,
\end{cases}
\end{equation}
where $\rho$ is the density of the elastic body, $u$ the displacement vector, $\sigma=\sigma(\theta, \nabla u)$ ($\theta$ is temperature) the stress tensor, $e=e(\theta, \nabla u)$ the internal energy and $q=q(\theta, \nabla \theta)$ the internal energy flux. They all can be represented by the Helmholtz free energy function $\psi:=e-\theta s$, where $s$ is the entropy. Moreover, $\sigma$ and $s$ are expressed as
\[
\sigma (\theta, \nabla u)=\psi_{F}(\theta, \nabla u),\;\;s(\theta, \nabla u)=-\psi_{\theta}(\theta, \nabla u),
\]
subscripts denote the partial derivatives ($\psi_F$ means the derivative with respect to the second variable).
The Helmholtz function representation of the parameters and 1D considerations lead \cite{slemrod} to the system \eqref{general}.

However, one can impose further structural assumptions on the Helmholtz function (for the details see \cite{TC2}), in particular,  the decoupling of energy
\[
e(\theta, \nabla u)=e_1(\theta)+e_2(\nabla u)
\]
and entropy
\[
s(\theta, \nabla u)=s_1(\theta)+s_2(\nabla u),
\]
together with the following prescription of $e_2, s_2$
\[
e_2(\nabla u)=\frac12|\nabla u|^2,\; s_2(\nabla u)=\mu\nabla\cdot u,
\]
which means that we consider only small deformations (here $\nabla\cdot u$ represents linearization of Jacobian).
Then we are led to the following simplified model for the balance of momentum and internal energy
\begin{equation}
\begin{cases}
\rho u_{tt}=\nabla \cdot (\nabla u-\mu\theta Id) ,&\textrm{in }(0,T)\times \Omega,\\
C_V(\theta)\theta_t+\nabla\cdot q=-\mu \theta \operatorname{div} u_t,&\textrm{in }(0,T)\times \Omega,
\end{cases}
\end{equation}
where $C_V:=\partial_{\theta}e_1(\theta)$ is the heat capacity. When assuming further
the density $\rho=1$ and the internal energy flux to satisfy the Fick law $q:=-\kappa(\theta)\nabla \theta$,
we arrive at the higher-dimensional extension of \eqref{eq}. Assuming further $\kappa=1$ and $C_V=1$ (the latter assumptions can be considered as an approximation of the high-temperature Dulong-Petit model), we obtain exactly
\eqref{eq} in higher dimensions.

\section{A priori estimates}\label{sec3}
The present section is devoted to the derivation of a crucial new estimate, which appears when estimating the evolution of the Fisher information of temperature along the trajectories of \eqref{eq}. Throughout this section we assume solutions of \eqref{eq} to be regular.
Rigorous treatment will be given in Section \ref{sec4}.

The first novel step is the following identity which holds along the trajectories of \eqref{eq}.
\begin{lem}\label{estlem}
Let us assume that $u\in C^3([0,T]\times[a,b])$ and $\theta\in C^{1,2}([0,T]\times[a,b])$, $\theta>0$ are solutions of (\ref{eq}). Then, the following identity holds
\begin{align*}
\frac12\ddt\left(\int_a^b\frac{\theta_x^2}{\theta}dx+\int_a^bu_{tx}^2dx+\int_a^bu_{xx}^2dx\right)=-\int_a^b\theta\left[(\log\theta)_{xx}\right]^2dx+\frac{\mu}2\int_a^b\frac{\theta^2_x}{\theta}u_{tx}dx.
\end{align*}
\end{lem}
\begin{proof}
We divide the second equation in \eqref{eq} by $\theta^{1/2}$ and  obtain
\begin{align}\label{rowlem1}
2\ddt\theta^{\sfrac12}-2(\theta^{\sfrac12})_{xx}-\frac12\frac{\theta_x^2}{\theta^{\sfrac32}}=\mu\theta^{\sfrac12}u_{tx}.
\end{align}
In addition, we have $\frac12\frac{\theta_x^2}{\theta^{\sfrac32}}=2\frac{(\theta^{\sfrac12})^2_x}{\theta^{\sfrac12}}$. Thus, we get
\begin{align}\label{row1lem1}
\ddt&\int_a^b\frac{\theta_x^2}{\theta}dx=4\ddt\int_a^b\left(\theta^{\sfrac1 2}\right)^2_xdx=8\int_a^b\left(\theta^{\sfrac1 2}\right)_x\left(\theta^{\sfrac1 2}\right)_{xt}dx\nonumber\\
&\stackrel{(\ref{rowlem1})}=4\iab\left(\theta^{\sfrac1 2}\right)_x\left[2(\theta^{\sfrac12})_{xx}+2\frac{(\theta^{\sfrac12})^2_x}{\theta^{\sfrac12}}+\mu\theta^{\sfrac12}u_{tx}\right]_x dx
=-8\iab\left[\left(\theta^{\sfrac1 2}\right)_{xx}\right]^2dx\nonumber\\
&\quad+16\iab\left(\theta^{\sfrac1 2}\right)_x\frac{\left(\theta^{\sfrac1 2}\right)_x\left(\theta^{\sfrac1 2}\right)_{xx}}{\theta^{\sfrac1 2}}dx-8\iab\frac{\left[\left(\theta^{\sfrac1 2}\right)_x\right]^4}{\theta}dx
+4\mu\iab\left(\theta^{\sfrac1 2}\right)_x\left[\theta^{\sfrac12}u_{tx}\right]_xdx.
\end{align}
Next, observe that
\begin{align}\label{row2lem1}
-8&\iab\left[\left(\theta^{\sfrac1 2}\right)_{xx}\right]^2dx
+16\iab\left(\theta^{\sfrac1 2}\right)_x\frac{\left(\theta^{\sfrac1 2}\right)_x\left(\theta^{\sfrac1 2}\right)_{xx}}{\theta^{\sfrac1 2}}dx-8\iab\frac{\left[\left(\theta^{\sfrac1 2}\right)_x\right]^4}{\theta}dx\nonumber\\
&=-8\iab\theta\left(\frac{\left(\theta^{\sfrac1 2}\right)_{xx}}{\theta^{\sfrac12}}-\frac{\left(\theta^{\sfrac1 2}\right)_x^2}{\theta}\right)^2dx=-8\iab\theta\left[\left(\log\theta^{\sfrac12}\right)_{xx}\right]^2dx\nonumber \\
&=-2\iab\theta\left[\left(\log\theta\right)_{xx}\right]^2dx.
\end{align}
Moreover, we calculate
\begin{align}\label{row3lem1}
4\mu\iab\left(\theta^{\sfrac12}\right)_x\left(\theta^{\sfrac12}u_{tx}\right)_x dx&=-4\mu\iab\left(\theta^{\sfrac12}\right)_{xx}\theta^{\sfrac12}u_{tx}dx\nonumber\\
&=-2\mu\iab\theta_{xx}u_{tx}dx+\mu\iab\frac{\theta_x^2}{\theta}u_{tx}dx,
\end{align}
where we used the identity $\left(\theta^{\sfrac12}\right)_{xx}=\frac12\theta^{-\tfrac{1}2}\theta_{xx}-\frac14\theta^{-\tfrac{3}2}\theta_x^2$.

Now, plugging (\ref{row2lem1}) and \eqref{row3lem1} in (\ref{row1lem1}) yields
\begin{align}\label{row5lem1}
\ddt\frac12\iab\frac{\theta^2_x}{\theta}dx=-\iab\theta\left[\left(\log\theta\right)_{xx}\right]^2dx-\mu\iab\theta_{xx}u_{tx}dx+\frac{\mu}2\iab\frac{\theta_x^2}{\theta}u_{tx}dx.
\end{align}
Next, we multiply the first equation on (\ref{eq}) by $u_{txx}$ and we integrate over $[a,b]$. Therefore, we obtain
\begin{align*}
\iab u_{tt}u_{txx}dx-\iab u_{xx}u_{txx}dx=\mu\iab\theta_xu_{txx}dx.
\end{align*}
This after integration by parts gives us
\begin{align}\label{row4lem1}
\frac12\ddt\left(\iab u_{tx}^2dx+\iab u_{xx}^2dx\right)=\mu\iab\theta_{xx}u_{tx}dx.
\end{align}
Then, we add (\ref{row4lem1}) and (\ref{row5lem1}) to obtain
\begin{align*}
\frac12\ddt\left(\iab u_{tx}^2dx+\iab u_{xx}^2dx+\iab\frac{\theta^2_x}{\theta}dx\right)=-\iab\theta\left[\left(\log\theta\right)_{xx}\right]^2dx+\frac{\mu}2\iab\frac{\theta_x^2}{\theta}u_{tx}dx.
\end{align*}
\end{proof}
The next step is to estimate the last term on the right-hand side of the identity in Lemma \ref{estlem}. This is done via the 1D embeddings as well as inequality (\ref{cfhslem}).
\begin{prop}\label{estprop1}
Let us assume that $u\in C^3([0,T]\times[a,b])$ and $\theta\in C^{1,2}([0,T]\times[a,b])$, $\theta>0$ are solutions of (\ref{eq}). Then for any $\epsilon>0$ and $\epsilon_1>0$ the following inequality
\begin{align}\label{thprop1}
\begin{split}
\frac{\mu}2\iab\frac{\theta_x^2}{\theta}u_{tx}dx&\leq \epsilon \iab\theta\left[\left(\log\theta\right)_{xx}\right]^2dx+c_1\|u_{tx}\|_{\lab{2}}^2\\
&\quad+\epsilon_1\left(\iab \left(\theta^{\sfrac12}\right)_x^2dx\right)^2+c_2\iab\left(\theta^{\sfrac12}\right)_x^2dx
\end{split}
\end{align}
holds and $c_1$ and $c_2$ depend on $\|u_0\|_{\habo}$, $\|\theta_0\|_{\lab1}$ $\|v_0\|_{\lab2}$, $a$, $b$, $\mu$, $\epsilon$ and $\epsilon_1$.
\end{prop}
\begin{proof}
Let us consider the left-hand side of (\ref{thprop1}).
\begin{align}\label{prop1in1}
\frac{\mu}2\iab\frac{\theta_x^2}{\theta}u_{tx}dx&=2\mu\iab\left(\theta^{\sfrac12}\right)_x^2u_{tx}dx=-4\mu\iab u_t\left(\theta^{\sfrac12}\right)_x\left(\theta^{\sfrac12}\right)_{xx}dx\nonumber\\
&\leq\frac{8\epsilon}{13}\iab\left[\left(\theta^{\sfrac12}\right)_{xx}\right]^2dx+c\iab u_t^2\left(\theta^{\sfrac12}\right)_x^2dx\nonumber\\
&\leq \epsilon \iab\theta\left[\left(\log\theta\right)_{xx}\right]^2dx+c\|u_t\| _{\lab{\infty}}^2\iab\left(\theta^{\sfrac12}\right)_x^2dx,
\end{align}
where we used Lemma \ref{cfhslem}.

Then, the Gagliardo-Nirenberg inequality yields
\begin{align*}
\|u_t\|_{\lab{\infty}}^2\iab\left(\theta^{\sfrac12}\right)_x^2dx&\leq c\|u_{t}\|_{\hab}\|u_t\|_{\lab{2}}\iab\left(\theta^{\sfrac12}\right)_x^2dx
\\&\stackrel{(\ref{balance})}{\leq} c_1\iab\left(\theta^{\sfrac12}\right)_x^2dx+{c}_2\|u_{tx}\|_{\lab{2}}\iab\left(\theta^{\sfrac12}\right)_x^2dx\\
&\leq c_1\iab\left(\theta^{\sfrac12}\right)_x^2dx+ c_2\|u_{tx}\|_{\lab{2}}^2+\epsilon_1\left(\iab\left(\theta^{\sfrac12}\right)_x^2dx\right)^2.
\end{align*}
Inserting the above inequality into (\ref{prop1in1}) gives the claim.
\end{proof}
Next, we handle the term $\left(\iab\left(\theta^{\sfrac12}\right)_x^2dx\right)^2$ utilizing Lemma \ref{cfhslem}.
\begin{prop}\label{estprop2}
The positive temperature $\theta$ solving \eqref{eq} satisfies
\begin{align*}
\left(\iab\left(\theta^{\sfrac12}\right)_x^2dx\right)^2\leq c\iab\theta\left[\left(\log\theta\right)_{xx}\right]^2dx,
\end{align*}
where constant $c$ depends on \eqref{balance}.
\end{prop}
\begin{proof}
We calculate
\begin{align*}
\left(\iab\left(\theta^{\sfrac12}\right)_x^2dx\right)^2&=\left(\iab\left(\theta^{\sfrac12}\right)_x\left(\theta^{\sfrac12}\right)_xdx\right)^2=\left(-\iab\left(\theta^{\sfrac12}\right)_{xx}\left(\theta^{\sfrac12}\right)dx\right)^2\\
&\leq\iab\theta dx\iab\left[\left(\theta^{\sfrac12}\right)_{xx}\right]^2dx\leq c\iab\theta\left[\left(\log\theta\right)_{xx}\right]^2dx,
\end{align*}
where in the last inequality we used Lemma \ref{cfhslem} and equality (\ref{balance}).
\end{proof}
We are now in a position to prove the main estimate of this section and a crucial ingredient of the present paper.
The following estimate is essential in arriving at further regularity estimates and plays a crucial role in both the existence and uniqueness parts of the main Theorem \ref{main_theorem}.

\begin{tw}\label{aprioritw}
For all $t\in [0,T]$ a regular solution $(u, \theta)$ to \eqref{eq}, with $\theta>0$, satisfies
\begin{align*}
\iab \left(\theta^{\sfrac12}\right)_x^2 dx+\iab u_{tx}^2dx+\iab u_{xx}^2dx\leq C_1 e^{C_2t},
\end{align*}
where $C_1$ and $C_2$ depend on $a$, $b$, $\mu$, $\|\theta_0^{\sfrac12}\|_{\hab}$, $\|u_0\|_{H^2(a,b)}$ and $\|v_0\|_{\habo}$.
\end{tw}
\begin{proof}
Let us take $0<\epsilon<\frac12$ in Proposition \ref{estprop1}. Then, plugging (\ref{thprop1}) in the inequality in Lemma \ref{estlem}, we obtain
\begin{align*}
\frac12\ddt&\left(\int_a^b\frac{\theta_x^2}{\theta}dx+\int_a^bu_{tx}^2dx+\int_a^bu_{xx}^2dx\right)
\\
&\leq-\frac12\int_a^b\theta\left[(\log\theta)_{xx}\right]^2dx+c_1\iab\left(\theta^{\sfrac12}\right)_x^2dx\\
&+ c_2\|u_{tx}\|_{\lab{2}}^2+\epsilon_1\left(\iab\left(\theta^{\sfrac12}\right)_x^2dx\right)^2.
\end{align*}
Now, by Proposition \ref{estprop2} and choosing $\epsilon_1>0$ in Proposition \ref{estprop1} such that $\epsilon_1c_3\leq\frac14$, we obtain
\begin{align*}
\frac12\ddt&\left(\int_a^b\frac{\theta_x^2}{\theta}dx+\int_a^bu_{tx}^2dx+\int_a^bu_{xx}^2dx\right)\\
&\leq-\frac14\int_a^b\theta\left[(\log\theta)_{xx}\right]^2dx+c_1\iab\left(\theta^{\sfrac12}\right)_x^2dx+ c_2\|u_{tx}\|_{\lab{2}}^2\\
&\leq c\left(\iab\left(\theta^{\sfrac12}\right)_x^2dx+\|u_{tx}\|_{\lab{2}}^2\right).
\end{align*}
Finally, the Gronwall inequality gives the claim.
\end{proof}
\section{Existence of solutions}\label{sec4}
The present section is devoted to the rigorous construction of a global-in-time regular solution to \eqref{eq}.
To this end the a priori estimate obtained in Theorem \ref{aprioritw} is used. However, a delicate construction of approximate solutions is required. In principle, there are many options for constructing solutions. It seems to us that the half-Galerkin procedure chosen by us leads to the final goal at the lowest price.

\subsection{Approximate problem}\label{approx}

In the wave equation part of \eqref{eq} we shall use the Galerkin approximation. The heat equation will then be solved in a straight manner. The whole procedure follows the scheme proposed in \cite{TC2}. We introduce a smooth basis $\{\phi_i\}$ of $\habo$ and denote $V_n=\operatorname{span}\{\phi_i\}_{1\leq i\leq n}$.
\begin{defi}\label{apdef}
We say that $u_n\in W^{1,\infty}(0,T;V_n)$ and $\theta_n\in L^2(0,T;H^1(a,b))$ are solutions of an approximate problem if for all $\phi\in V_n$ and $\psi\in H^1(a,b)$ the following equations are satisfied in $\mathcal{D}'(0,T)$
\begin{align*}
\frac{d^2}{dt^2}\iab u_n\phi dx+\iab u_{n,x}\phi_x dx&=-\mu\iab\theta_n\phi_x dx,\\
\ddt\iab\theta_n\psi dx+\iab \theta_{n,x}\psi_xdx&=\mu\iab\theta_n u_{n,tx}\psi dx.
\end{align*}
Moreover, the following equalities
\begin{align*}
\theta_n(0)=\theta_{0},\quad u_n(0)=P_{V_n}u_0,\quad u_{n,t}(0)=P_{V_n}v_0
\end{align*}
are satisfied. Where $P_{V_n}$ is an orthogonal projection of the corresponding spaces onto $V_n$.
\end{defi}
By \cite[Proposition 1 and Lemma 4.2]{TC2} we have the following result
\begin{prop}\label{aproxex}
For every $n\in\mathbb{N}$ there exists a solution $(u_n,\theta_n)$ of the approximate problem in the sense of Definition \ref{apdef}. Moreover, one has $\theta_n\in L^2(0,T;H^2(a,b))\cap H^1(0,T;L^2(a,b))$ and there exist $\hat\theta_n>0$ such that $\theta_n(x,t)\geq\hat\theta_n$ for almost all $(t,x)\in [0,T]\times[a,b]$.
\end{prop}

\subsection{Passing to the limit}
In this section, we show that the approximations defined in Section \ref{approx} converge to the solution of \eqref{eq} in the sense of Definition \ref{weakdef}.
\begin{tw}\label{twosz}
If $(u_n,\theta_n)$ is the solution of the approximate problem in the sense of Definition \ref{apdef}, then the following estimates (independent of $n$) are satisfied:
\begin{eqnarray}
&\|u_n\|_{W^{1,\infty}(0,T;\habo)}\leq C, \;\;\|u_n\|_{L^{\infty}(0,T;H^2(a,b))}\leq C,\label{pierwsze}\\ &\|u_{n,tt}\|_{L^\infty(0,T;\lab{2})}\leq C,\label{drugie}\\
&\|\theta_n\|_{L^{\infty}(0,T;\hab)}\leq C,\label{trzecie}\\
&\|\theta_{n,t}\|_{L^2(0,T;\lab{2})}\leq C,\label{czwarte}
\end{eqnarray}
where $C=C(a, b,\mu, T,\|\theta_0\|_{\hab}, \tilde\theta,\|u_0\|_{H^2(a,b)}, \|v_0\|_{\habo})$.
\end{tw}
\begin{proof}
We know that $u_n\in V_n$ and $V_n\subset C^{\infty}[a,b]$. Thus, we obtain $u_n(t)\in C^{\infty}[a,b]$ for almost all $t\in [0,T]$. On the other hand, by standard theory for hyperbolic equations (see \cite{evans}) we have $u_n\in C^{3,\infty}([0,T]\times[a,b])$. Similarly, by standard regularity theory of parabolic equations, we have $\theta_n\in C^{1,2}([0,T]\times[a,b])$ (see \cite[Chapter VII]{lady} or \cite[Chapter 5]{frid}). Therefore, identities in Definition \ref{apdef} are satisfied pointwise for all $t\in [0,T]$
\begin{align}
\iab u_{n,tt}\phi dx+\iab u_{n,x}\phi_x dx&=-\mu\iab\theta_n\phi_x dx,\label{apdef1}\\
\iab\theta_{n,t}\psi dx+\iab \theta_{n,x}\psi_xdx&=\mu\iab\theta_n u_{n,tx}\psi dx.\label{apdef2}
\end{align}
Moreover, our approximations $(u_n, \theta_n)$ are regular and they satisfy exactly \eqref{eq}, so that Theorem \ref{aprioritw} as well as Proposition \ref{energy_cons} can be applied directly to $(u_n, \theta_n)$ and yield
\begin{align}\label{minin1}
\|u_n\|_{W^{1,\infty}(0,T;\habo)}+\|u_n\|_{L^{\infty}(0,T;H^2(a,b))}+\|(\theta_n)^{\sfrac12}\|_{L^{\infty}(0,T;\hab)}\leq C.
\end{align}
The above gives \eqref{pierwsze}.

Next, let us take $\psi=\partial_t(\theta_n)$ in (\ref{apdef2}) and integrate the given equality over $[0,t]$ for an arbitrary $t\in (0,T]$. We obtain
\begin{align*}
\int_0^t\iab(\partial_t\theta_n)^2dxdt+\frac12\iab(\partial_x\theta_n)^2dx=\frac12\|(\partial_x\theta_{0})\|^2_{\lab{2}}+\mu\int_0^t\iab\theta_n(\partial_t\theta_n) (\partial^2_{tx}u_n)dxdt.
\end{align*}
It gives
\begin{align*}
\|\partial_t\theta_n\|_{L^2(0,T;\lab{2})}^2&+\frac12\|\partial_x\theta_n\|^2_{\lab{2}}\leq\frac12\|\partial_x(\theta_{0})\|^2_{\lab{2}}\\
&\quad+|\mu|\|\theta_n\|_{L^{\infty}(0,T;\lab{\infty})}\|\partial_t\theta_n\|_{L^2(0,T;\lab{2})}\|\partial_{tx}^2u_n\|_{L^2(0,T;\lab{2})}\\
&\leq\frac12\|\partial_x(\theta_{0})\|^2_{\lab{2}}+\frac12\|\partial_t\theta_n\|_{L^2(0,T;\lab{2})}^2+c\|\partial_{tx}^2u_n\|^2_{L^2(0,T;\lab{2})},
\end{align*}
where in the last inequality we used \eqref{minin1} and an embedding
\[
\|(\theta_n)^{\sfrac12}\|_{L^\infty\left((0,T);L^\infty(a,b)\right)}\leq C \|(\theta_n)^{\sfrac12}\|_{L^\infty\left((0,T);H^1(a,b)\right)},
\]
which yields an $n$ independent estimate of $\theta_n$ in $L^\infty$.

Hence, we have
\begin{align}\label{minin2}
&\|\partial_t\theta_n\|_{L^2(0,T;\lab{2})}^2+\|\partial_x\theta_n\|^2_{L^{\infty}(0,T;\lab{2})}\leq\nonumber \\ &\|\partial_x(\theta_{0})\|^2_{\lab{2}}+c\|\partial^2_{tx}u_n\|^2_{L^2(0,T;\lab{2})}\leq c,
\end{align}
where again we used \eqref{minin1} in the last inequality. Hence, both \eqref{trzecie} and \eqref{czwarte} are satisfied.

Finally, let us take $\psi=u_{n,tt}$ in (\ref{apdef1}). We obtain
\begin{align*}
\iab (u_{n,tt})^2 dx&=\iab u_{n,xx}u_{n,tt}dx+\mu\iab\theta_{n,x}u_{n,tt} dx\\
&\leq \frac12\iab (u_{n,tt})^2 dx+c\left(\|u_{n,xx}\|^2_{\lab{2}}+\|\theta_{n,x}\|^2_{\lab{2}}\right).
\end{align*}
Thus, in light of (\ref{minin1}) and (\ref{minin2}), we have
\begin{align*}
\|u_{n,tt}\|_{L^\infty(0,T;\lab{2})}\leq c.
\end{align*}

In the end let us note that the constant in \eqref{minin1} depends on $\|\theta_0^{\sfrac12}\|_{\hab}$. But we have the following inequality
\begin{align*}
\|\theta_0^{\sfrac12}\|_{\hab}\leq\left((b-a)^{\sfrac12}\|\theta_0\|_{\lab{2}}+\frac1{\tilde{\theta}}\int_a^b\theta_{0,x}^2dx\right)^{\sfrac12}.
\end{align*}
\end{proof}

Next, we proceed to prove the existence of a solution to \eqref{eq}, and hence the existence part of Theorem \ref{main_theorem}. Indeed, we have the following.
\begin{tw}
For any $T>0$, there exists a solution to problem (\ref{eq}) as described in Definition \ref{weakdef}.
\end{tw}
\begin{proof}
By Theorem \ref{twosz} there exist a subsequence of $(u_n,\theta_n)$ (still denoted as $(u_n,\theta_n)$) and
\begin{align*}
u&\in L^{\infty}(0,T;H^2(a,b))\cap W^{1,\infty}(0,T,\habo)\cap W^{2,\infty}(0,T;\lab{2}),\\
\theta&\in L^{\infty}(0,T;\hab)\cap H^1(0,T; \lab{2})
\end{align*}
such that the following weak convergences
\begin{align}\label{weakex}
\begin{split}
u_n&\stackrel *{\rightharpoonup}u\quad\ \;  \textrm{ in }L^{\infty}(0,T;H^2(a,b)),\\
u_{n,t}&\stackrel *{\rightharpoonup}u_t\quad\ \textrm{ in }L^{\infty}(0,T;\habo),\\
u_{n,tt}&\stackrel *{\rightharpoonup}u_{tt}\quad \textrm{ in }L^{\infty}(0,T;\lab{2}),\\
\theta_{n}&\stackrel*{\rightharpoonup}\theta\quad\ \ \textrm{ in }L^{\infty}(0,T;\hab),\\
\theta_{n,t}&\stackrel{}{\rightharpoonup}\theta_t\quad\ \textrm{ in }L^2(0,T;\lab{2})\end{split}
\end{align}
hold.

Next, by the Gelfand triple theorem (see for instance \cite{evans}), we notice that
\[
\theta\in C([0,T];L^2(a,b)),
\]
\[
u\in C([0,T];\habo)\;\;\mbox{and}\;\;u_t\in C([0,T];L^2(a,b)).
\]
In the standard way, we show that $u$ and $\theta$ satisfy the momentum equation in Definition \ref{weakdef}. Now, let us take $\psi\in C^{\infty}[a,b]$ and $\phi\in C^{\infty}[0,T]$. We multiply the second equation in Definition \ref{apdef} and integrate over $[0,T]$. It yields
\begin{align*}
\iot\iab\theta_{n,t}\psi\phi dxdt+\iot\iab \theta_{n,x}\psi_x\phi dxdt=\mu\iot\iab\theta_n u_{n,tx}\psi\phi dxdt.
\end{align*}
It is easy to see that
\begin{align*}
\iot\iab\theta_{n,t}\psi\phi dx dt+\iot\iab\theta_{n,x}\psi_x\phi dx dt 
\to \iot\iab\theta_t\psi\phi dx dt+\iot\iab\theta_x\psi_x\phi dx dt.
\end{align*}
It is left to show that
\begin{align}\label{conv}
\iot\iab\theta_n u_{n,tx}\psi\phi dx dt\to\iot\iab\theta u_{tx}\psi\phi dx dt.
\end{align}
The Aubin-Lions lemma (see \cite{aubin,rub}) and convergences from (\ref{weakex}) give that
\begin{align}\label{strcon}
 \theta_n\to\theta\quad\textrm{ in }L^{2}(0,T;\lab{2}).
 \end{align}
 Thus, we get that (\ref{conv}) holds.
\end{proof}
In the following remark, we emphasize that the obtained solution (besides having enough regularity to define them time globally) is bounded in $L^\infty$. 
\begin{rem}
Let us note that there exist constants $C_1$ and $C_2$ such that the following inequalities are satisfied
\begin{align*}
\|u(t,\cdot)\|_{\lab\infty}&\leq C_1e^{C_2t},\ \|u_t(t,\cdot)\|_{\lab\infty}\leq C_1e^{C_2t},\\
 &\quad\|\theta(t,\cdot)\|_{\lab\infty}\leq C_1e^{C_2t}.
\end{align*}
\end{rem}
\subsection{Positivity of the temperature}
The claim of Theorem \ref{main_theorem} is that $\theta$ is positive. The present section is devoted to the proof
of the latter.
\begin{tw}\label{positivity+}
Let $(u,\theta)$ be a solution of (\ref{eq}) in the meaning of Definition \ref{weakdef}. Then there exists $\hat\theta>0$ such that the temperature function satisfies
$$
\theta(t,x)\geq\hat\theta
$$
for almost $(t,x)\in[0,T]\times[a,b]$.
\end{tw}
The proof splits into two steps. The first is contained in the following lemma.
\begin{lem}\label{logarytmy}
If $\theta_n$ are the approximate solutions in the sense of Definition \ref{apdef}, then
there exists $C>0$ such that the following estimate is satisfied
\begin{align}
\|\log\theta_n\|_{\lab{\infty}(0,T;\hab)}&\leq C,\label{logar3}\\
\|\left(\log\theta_n\right)_t\|_{L^{\infty}(0,T;\left(\hab\right)^*)}&\leq C,\label{logar2}
\end{align}
where $C=C(a, b,\mu, T,\|\theta_0\|_{\hab}, \tilde\theta,\|u_0\|_{H^2(a,b)}, \|v_0\|_{\habo})$.
\end{lem}
\begin{proof}
The proof makes use of the energy introduced in \cite{TC2}. From Proposition \ref{aproxex} we know that
\begin{align*}
\theta_n(t,x)>0
\end{align*}
for all $(t,x)\in[0,T]\times[a,b]$ and $n\in\mathbb{N}$. By (\ref{apdef2}) we  get that
\begin{align*}
\theta_{n,t}-\theta_{n,xx}=\mu\theta_nu_{n,tx}
\end{align*}
holds for all $(t,x)\in[0,T]\times[a,b]$. We divide the above identity by $\theta_n$ to arrive at
\begin{align}\label{taueq}
\tau_{n,t}-\tau^2_{n,x}-\tau_{n,xx}=\mu u_{n,tx},
\end{align}
where we denote $\tau_n:=\log\theta_n$.

We integrate \eqref{taueq} over $[a,b]$ and obtain
\begin{align*}
\iab\tau_{n,t}dx-\iab\tau_{n,x}^2dx=0.
\end{align*}
Now, we integrate the above equality over $[0,T]$. It yields
\begin{align}\label{posin1}
\iab\tau_ndx=\iab\tau_0dx+\iot\iab\tau_{n,x}^2dxdt,
\end{align}
where $\tau_0=\log\theta_0$.
Subtracting (\ref{posin1}) from (\ref{balance}) we observe
\begin{align*}
\frac12&\iab u_{n,t}^2dx+\frac12\iab u_{n,x}^2dx+\iab(\theta_n-\tau_n) dx+\iot\iab\tau_{n,x}^2dxdt\\
&=\frac12\iab v_{n,0}^2dx+\frac12\iab u_{n,0,x}^2dx+\iab(\theta_{0} -\tau_{0})dx
\end{align*}
for all $n\in\mathbb N$. It gives us
\begin{equation}\label{logar4}
\|\tau_n\|_{L^\infty(0,T;L^1(a,b))}\leq C.
\end{equation}

Next, let us multiply \eqref{taueq} by $\tau_{n,xx}$ and integrate over $[a,b]$. It yields
\begin{align*}
\iab\tau_{n,t}\tau_{n,xx}dx-\iab\tau^2_{n,x}\tau_{n,xx}dx-\iab\tau^2_{n,xx}dx=\mu\iab u_{n,tx}\tau_{n,xx}dx.
\end{align*}
After integration by parts we get
\begin{align*}
-\frac12\ddt\iab\tau_{n,x}^2dx-\iab\tau^2_{n,xx}dx=\mu\iab u_{n,tx}\tau_{n,xx}dx.
\end{align*}
This gives us
\begin{align*}
\frac12\ddt\|\tau_{n,x}\|^2_{\lab2}+\|\tau_{n,xx}\|_{\lab2}^2\leq c\|u_{n,tx}\|^2_{\lab2}+\frac12\|\tau_{n,xx}\|^2_{\lab2}.
\end{align*}
Thus, by Theorem \ref{twosz}, \eqref{pierwsze} we get
\begin{align*}
\ddt\|\tau_{n,x}\|^2_{\lab2}\leq C.
\end{align*}
Integrating over $[0,t]$ and recalling \eqref{logar4}, we obtain \eqref{logar3}.

Now, let us take $w\in \hab$ such that $\|w\|_{\hab}=1$. We multiply \eqref{taueq} by $w$ and integrate over $[a,b]$ to arrive at
\begin{align*}
\iab\tau_{n,t}wdx&=\iab\tau_{n,x}^2wdx+\iab\tau_{n,xx}wdx+\mu\iab u_{n,tx}wdx\\
& =\iab\tau_{n,x}^2wdx-\iab\tau_{n,x}w_xdx+\mu\iab u_{n,tx}wdx\\
&\leq \|\tau_n\|_{\hab}^2\|w\|_{\lab\infty}+\|\tau_{n,x}\|_{\lab2}\|w_x\|_{\lab2}+|\mu|\|u_{tx}\|_{\lab2}\|w\|_{\lab2}\\
&\leq C\|\tau_n\|_{\hab}^2+\|\tau_{n,x}\|_{\lab2}+|\mu|\|u_{tx}\|_{\lab2}.
\end{align*}
This, thanks to Theorem \ref{twosz} and \eqref{logar3}, gives us \eqref{logar2}.
\end{proof}

Now, we go back to the proof of Theorem \ref{positivity+}

\begin{proof}
We still denote $\tau_n=\log\theta_n$. By Lemma \ref{logarytmy}
$\tau_n$ is bounded in $L^\infty(0,T;\hab)$ and $\tau_{n,t}$ is bounded in $L^{\infty}(0,T;(\hab)^*)$. Thus, by Aubin--Lions lemma we get that there exists a subsequence of $\tau_n$ and $\tau\in L^2(0,T;\lab2)$ such that
\begin{align*}
\tau_n\to\tau\quad\textrm{ in }L^2(0,T;\lab2).
\end{align*}
It yields that (possibly on subsequence)
\begin{align}\label{aetau}
\tau_n(t,x)\to\tau(t,x)
\end{align}
for almost all $(t,x)\in [0,T]\times[a,b]$.
On the other hand, (\ref{strcon}) gives us that
\begin{align}\label{aetheta}
\theta_n(t,x)\to\theta(t,x)
\end{align}
for almost all $(t,x)\in [0,T]\times[a,b]$. Because $\theta_n=e^{\tau_n}$, thanks to \eqref{aetau} and \eqref{aetheta}, we get
\begin{align*}
e^{\tau}=\theta.
\end{align*}

Now, by \eqref{logar3} we obtain
\begin{align*}
\tau_{n}&\stackrel*{\rightharpoonup}\tau\quad \textrm{ in }L^{\infty}(0,T;\hab).
\end{align*}
Therefore, we have that $\tau\in L^{\infty}(0,T;\hab)$. It eventually yields that $\tau\in L^{\infty}(0,T;\lab\infty)$. This implies that
\begin{align*}
e^{-\|\tau\|_{L^{\infty}(0,T;\lab\infty)}}\leq \theta.
\end{align*}
\end{proof}

\section{Uniqueness of solutions}\label{sec5}

In this section, the uniqueness claim of Theorem \ref{main_theorem} is achieved.
We prove the following result.
\begin{tw}
The solution satisfying (\ref{eq}) in the sense of Definition \ref{weakdef} is unique.
\end{tw}
\begin{proof}
Let us assume that $(u_1,\theta_1)$ and $(u_2,\theta_2)$ are two solutions of (\ref{eq}) in the sense of Definition \ref{weakdef}. Let us also denote $u=u_1-u_2$ and $\theta=\theta_1-\theta_2$. Now, let us subtract the momentum equation for $(u_2,\theta_2)$ from the momentum equation for $(u_1,\theta_1)$. In addition, w multiply the obtained identity by $u_t$ and integrate over $[a,b]$. It yields
\begin{align*}
\iab u_{tt}u_t dx+\iab u_xu_{tx} dx=\mu\iab\theta_xu_t dx.
\end{align*}
Hence,
\begin{align}\label{uniin2}
\frac12\ddt\left(\|u_t\|_{\lab{2}}^2+\|u_x\|_{\lab{2}}^2\right)\leq\epsilon\|\theta_x\|_{\lab{2}}^2+c\|u_t\|_{\lab{2}}^2.
\end{align}

Similarly, we proceed with the entropy equation and test it by $\psi=\theta$, then
\begin{align*}
\iab\theta_t\theta dx&+\iab\theta_x^2 dx =\mu\iab\theta_1 u_{1,tx}\theta dx-\mu\iab\theta_2 u_{2,tx}\theta dx\\
&=\mu\iab\theta^2 u_{1,tx} dx+\mu\iab\theta_2 u_{tx}\theta dx.
\end{align*}
We integrate by parts the right-hand side of the above equality to obtain
\begin{align}\label{uniin1}
\frac12\ddt\|\theta\|_{\lab{2}}^2+\|\theta_x\|_{\lab{2}}^2&=-2\mu\iab\theta\theta_xu_{1,t}dx-\mu\iab u_t\theta_{2,x}\theta dx-\mu\iab u_t\theta_2\theta_x dx\nonumber \\
&:=I_1+I_2+I_3.
\end{align}
Next, we estimate all the terms $I_1,I_2,I_3$ in a similar manner, using the one-dimensional embeddings.
\begin{align*}
I_1\leq 2\mu\|u_{1,t}\|_{\lab{\infty}}\|\theta\|_{\lab{2}}\|\theta_x\|_{\lab{2}}\leq c(\epsilon)\|\theta\|_{\lab{2}}^2+\epsilon\|\theta_x\|_{\lab{2}}^2,
\end{align*}
where in the last inequality we used the embedding $H^1\subset L^\infty$ and \eqref{weakex}.

With $I_2$ we proceed as follows
\begin{align*}
I_2&\leq \mu\|u_t\theta\|_{\lab{2}}\|\theta_{2,x}\|_{\lab{2}}\leq \mu\|\theta\|_{\lab{\infty}}\|u_t\|_{\lab{2}}\|\theta_{2,x}\|_{\lab{2}}\\
&\leq c\|\theta\|_{\hab}\|u_t\|_{\lab{2}}
\leq \epsilon\|\theta_x\|_{\lab{2}}^2+\epsilon\|\theta\|_{\lab2}^2+C(\epsilon)\|u_t\|^2_{\lab{2}},
\end{align*}
where in the third inequality we used the fact that $\theta_2$ satisfies \eqref{uniq}.
Next, we estimate $I_3$
\begin{align*}
I_3\leq\mu\|\theta_2\|_{\lab{\infty}}\|u_t\|_{\lab2}\|\theta_x\|_{\lab2}\leq \epsilon\|\theta_x\|^2_{\lab2}+c(\epsilon)\|u_t\|_{\lab2}^2,
\end{align*}
again we used the fact that $\theta_2$ satisfies \eqref{uniq} and embedding $H^1(a,b)\subset L^\infty(a,b)$.

Next, we plug in the above estimates of $I_1, I_2$ and $I_3$ in  (\ref{uniin1}) and obtain
\begin{equation}\label{prawie}
\frac12\ddt\|\theta\|_{\lab{2}}^2+\|\theta_x\|_{\lab{2}}^2\leq3\epsilon\|\theta_x\|_{\lab2}^2+c(\epsilon)\left(\|\theta\|_{\lab2}^2+\|u_t\|_{\lab2}^2\right).
\end{equation}
Adding \eqref{prawie} and (\ref{uniin2}) we arrive at
\begin{align*}
\frac12\ddt\left(\|u_t\|_{\lab{2}}^2+\|u_x\|_{\lab{2}}^2+\|\theta\|_{\lab2}^2\right)+\|\theta_x\|_{\lab2}^2\\
\leq 4\epsilon\|\theta_x\|_{\lab2}^2+c(\epsilon)\left(\|\theta\|_{\lab2}^2+\|u_t\|_{\lab2}^2\right).
\end{align*}
Now, let us take $\epsilon=\frac18$, then
\begin{align*}
\frac12\ddt\left(\|u_t\|_{\lab{2}}^2+\|u_x\|_{\lab{2}}^2+\|\theta\|_{\lab2}^2\right)\leq c\left(\|\theta\|_{\lab2}^2+\|u_t\|_{\lab2}^2\right),
\end{align*}
and by the Gronwall inequality $\theta=u_x=u_t=0$ for almost all $t\in [0,T]$. Hence, we immediately get $u_1=u_2$ and $\theta_1=\theta_2$.
\end{proof}

\section{Conclusions}\label{sec6}
In our article, we have proved the existence of a global-in-time regular unique solution to the minimal nonlinear
thermoelasticity problem \eqref{eq}. Our solution is based on a new estimate, described in Section \ref{sec3}.
The interesting problem for future research is the possibility of extending our estimates to the more general,
thermodynamically relevant system \eqref{general}. Since Fisher's information has a deep thermodynamical meaning, our approach seems to have a chance to be successful also in more general thermodynamically relevant problems. In particular, we hope that an extension of our Fisher's entropy estimates coupled with a proper generalization of an equality \cite{CFHS} would yield global-in-time solution to \eqref{general} for some choices of $\alpha,\beta,\gamma,\delta$. We emphasize that our approach cannot work for general $\alpha,\beta,\gamma,\delta$ due to the singularity occurrence in \cite{daf_hsiao}. Finding sufficient conditions on $\alpha,\beta,\gamma,\delta$ guaranteeing global solutions is an issue. 

Next, \eqref{eq} can be treated as a linearization of \eqref{general}, like in \cite{racke_shibata}, and this way yield essential information on more general models described by \eqref{general}.  As we mentioned above, one could
hope to distinguish between models guaranteeing global existence from those allowing singularities, the first step towards a criterion would involve looking for functionals similar to the one in Lemma \ref{estlem}. The present paper needs to be treated as a first step of a long-term programme.

The next issue is a problem with a global-in-time unique solution in higher dimensions. Let us recall the minimal version
of \eqref{eq} in physically relevant dimensions $2$ and $3$. This time temperature is a real-valued function $\theta$
defined on a bounded subset $\Omega$. Displacement is a vector field $u\colon\Omega\to\mathbb R^n$, $n=2,3$. Temperature satisfies zero Neumann boundary condition, while displacement satisfies zero Dirichlet boundary condition. Moreover, they satisfy
\begin{equation}\label{eq_dim}
\begin{cases}
u_{tt}-\Delta u=\mu\nabla \theta,&\textrm{in }(0,T)\times \Omega,\\
\theta_t-\theta_{xx}=\mu\theta \operatorname{div} u_t,&\textrm{in }(0,T)\times \Omega,\\
u(0,\cdot)=u_0,\ u_t(0,\cdot)=v_0,\ \theta(0,\cdot)=\theta_0>0.
\end{cases}
\end{equation}
In \cite{TC2} the global-in-time measure weak solution with a defect measure is obtained in 3D (the same construction
yields also a solution in 2D). Such a solution satisfies a weak-strong uniqueness. Moreover, $\theta$
is a positive measure. The question is whether finite-time singularities occur or not? Is there a difference
between dimension $2$ and $3$? It seems that a 2D case offers more structure. Notice that the vector field $u$ can be handled easier in 2D in view of the Helmholtz decomposition. On the other hand, the 2D formulation of \eqref{eq_dim} as a hyperbolic system seems also easier to deal with in 2D. The latter should be helpful in obtaining either the global existence results or the blowup ones. 

These are again interesting problems. Notice that our identity in Lemma \ref{estlem} can be extended to the higher-dimensional cases. We state it below. Notice that in 2D it seems to lead to the border case estimates.

We have the following formal estimate (we assume that solutions are regular).
\begin{prop}
Solution $(u, \theta)$ to \eqref{eq_dim} satisfies
\begin{align*}
\frac{1}{2}\frac{d}{dt}&\left(\int_\Omega \frac{|\nabla \theta|^2}{\theta}dx+\int_\Omega|\operatorname{div} (u_t)|^2dx+\int_\Omega|\nabla \operatorname{div} u|^2dx\right)=\\
&-\int_\Omega \theta\left|D^2 \log \theta\right|^2dx+\frac{\mu}{2}\int_\Omega \frac{|\nabla \theta|^2}{\theta}\operatorname{div}(u_t)dx.
\end{align*}
\end{prop}
\begin{proof}
The proof is analogous to the proof of Lemma \ref{estlem}, so we only sketch it. First, the higher-dimensional Fisher information argument gives
\[
\frac12\frac{d}{dt}\int_\Omega \frac{|\nabla \theta|^2}{\theta}dx=-\int_\Omega \theta|D^2 \log \theta|^2dx-\mu\int_\Omega \Delta \theta \operatorname{div}(u_t)dx+\frac{\mu}{2}\int_\Omega \frac{|\nabla\theta|^2}{\theta}\operatorname{div}(u_t)dx.
\]
Next, multiplying the wave equation in \eqref{eq_dim} by $\nabla \operatorname{div}(u_t)$, we arrive at
\[
\frac{d}{dt}\left(\frac12\int_\Omega |\operatorname{div}(u_t)|^2dx+\frac12 \int_\Omega |\nabla \operatorname{div}u|^2dx\right)=\mu\int_\Omega \Delta \theta \operatorname{div}(u_t)dx.
\]
The claim follows.
\end{proof}

{\bf Acknowledgement.} T.C.\ was supported by the National Science Center of Poland grant SONATA BIS 7 number UMO-2017/26/E/ST1/00989.

\end{document}